\documentclass[11pt]{amsart}
\usepackage{amsmath,amsthm,amssymb,enumerate,natbib}

\newcommand{\prob}{\stackrel{P}{\longrightarrow}}
\newcommand{\fdd}{\stackrel{f.d.d.}{\longrightarrow}}
\newcommand{\eid}{\stackrel{d}{=}}

\newcommand{\reals}{{\mathbb R}}
\newcommand{\bbr}{\reals}
\newcommand{\vep}{\varepsilon}

\newtheorem{theorem}{Theorem}[section]

\def\Cov{{\rm Cov}}

\def\Var{{\rm Var}}
\def\E{{\rm E}}

\numberwithin{equation}{section}
\begin{document}
%%%%%%%%%%%%%%%%%%%%%%%%%
\title[Length of stationary Gaussian excursions]
{Length of stationary Gaussian excursions}

\author{Arijit Chakrabarty}
\address{Theoretical Statistics and Mathematics Unit \\
Indian Statistical Institute \\
203 B.T. Road\\
Kolkata 700108, India}
\email{arijit.isi@gmail.com}

\author{Manish Pandey}
\address{Department of Mathematics and Computer Science, Eindhoven University of Technology\\}
\email{manishpandey0897@gmail.com}

\author{Sukrit Chakraborty}
\address{Theoretical Statistics and Mathematics Unit \\
Indian Statistical Institute \\
203 B.T. Road\\
Kolkata 700108, India}
\email{sukrit049@gmail.com}

\subjclass{Primary 60G15.  Secondary 60G70.}
\keywords{ Gaussian process,  high excursions, length of excursion,  regular variation
\vspace{.5ex}}

\begin{abstract}
Given that a stationary Gaussian process is above a high threshold, the length of time it spends before going below that threshold is studied.  The asymptotic order is determined by the smoothness of the sample paths, which in turn is a function of the tails of the spectral measure.  Two disjoint regimes are studied - one in which the second spectral moment is finite and the other in which the tails of the spectral measure are regularly varying and the second moment is infinite.
\end{abstract}

\maketitle

\section{Introduction} 

Gaussian processes have been studied from various points of view, including from the angle of extremal behaviour.  A list,  which is incomplete by all means, of works on supremum and infimum of Gaussian processes include \cite{dudley:1973},  \cite{talagrand:1987},  \cite{berman:kono:1989}, besides the books \cite{adler:taylor:2007} and \cite{azais:wschebor:2009}.  Examples of recent works on this topic are \cite{adler:moldavskaya:samorodnitsky:2014},  \cite{chakrabarty:samorodnitsky:2017},  \cite{wu:chakrabarty:samorodnitsky:2019},  \cite{selk:2021} among others.  

In this short note,  we study the length of excursion sets of a stationary Gaussian process above a high threshold.  Consider a stationary zero mean Gaussian process $(X_t:t\in\bbr)$ with continuous sample paths, the exact assumptions on which are mentioned in Section \ref{sec1}.  Given that at some point, which we take as the origin, the process is above a high threshold $u$,  the excursion set is defined as the connected component of the random set $\{t\in\bbr:X_t>u\}$ containing the origin.  The asymptotic order of the excursion set, given that $X_0>u$,  for large $u$, depends on the smoothness of the sample paths of $(X_t)$. Potential applications of the results proved herein are in statistical analyses in situations where data is available only above a very high threshold. In such situations, smoothness of the paths can be inferred from the lengths of the excursion sets.

The smoothness of the sample paths, in turn,  is determined by the tails of the spectral measure of the covariance kernel underlying $(X_t)$.  We study two disjoint regimes.  In Section \ref{sec.c2}, the spectral measure is considered to have a finite second moment.  In this regime, the paths of $(X_t)$ are twice continuously differentiable, or $C^2$ as we usually say.  In Section \ref{sec.ht}, the spectral measure is assumed to have regularly varying tails with index $-\alpha$ for some $0<\alpha<2$, which automatically ensures that the second spectral moment is infinite.  The analysis of the 2 regimes proceeds along very different lines.  In the $C^2$ regime,  differential calculus is used for studying the length of the excursions.  Theorem \ref{sec.c2.t1} shows the excursion set above $u$ is asymptotically of the order $u^{-1}$.  The study in Section \ref{sec.ht} proceeds by scaling time and space at an appropriate level with the help of a result of \cite{pitman:1968}.  It is shown in Theorem \ref{sec.ht.main} that the asymptotic order is approximately $u^{-2/\alpha}$, which is much smaller than that in the $C^2$ regime.  This is not surprising in view of the fact that the sample paths are less smooth when the spectral measure has regularly varying tails.  The results are applied to an example in Section \ref{sec.ex}.

\section{The assumptions}\label{sec1}
Let $\mu$ be a finite non-null symmetric measure on $\bbr$, that is,  
\[
0<\mu(\bbr)<\infty\,,
\]
and for every Borel set $A\subset\bbr$,
\[
\mu(-A)=\mu(A)\,,
\]
where $-A=\{x\in\bbr:-x\in A\}$.  Define
\begin{equation}
\label{eq.defR}R(t)=\int_{-\infty}^\infty \mu(dx)\,e^{itx}\,,t\in\bbr\,,
\end{equation}
where $i=\sqrt{-1}$, the right hand side being a real number due to symmetry of $\mu(\cdot)$.  The following assumption is made throughout the paper.  For some $\vep>0$,
\begin{equation}
\label{sec1.eq1} R(t)=R(0)+O\left(|t|^\vep\right)\,,t\in\bbr\,.
\end{equation}

Let $(X_t:t\in\bbr)$ be a zero mean Gaussian process with
\[
\E\left(X_sX_t\right)=R(s-t)\,,s,t\in\bbr\,,
\]
\eqref{eq.defR} ensuring the existence of such a process.  Furthermore, $(X_t)$ can be chosen to have continuous paths due to \eqref{sec1.eq1}.  Without loss of generality, $(X_t:t\in\bbr)$ is thus a zero mean stationary Gaussian process with covariance kernel $R(\cdot)$ and continuous paths.

The goal of this short note is studying the length of the excursion set above a high threshold $u$, given that the process exceeds $u$ at some point, which we assume to be the origin.  That is, for all $u>0$, setting
\begin{align}
\label{sec1.eq2}\tau_u^+&=\inf\{t\ge0:X_t\le u\}\,,\\
\label{sec1.eq3}\tau_u^-&=\sup\{t\le0:X_t\le u\}\,,
\end{align}
we are interested in understanding the asymptotic distribution of $\tau_u^+-\tau_u^-$, after suitable scaling,  as $u\to\infty$, conditionally on $[X_0>u]$.  

\section{Finite second spectral moment}\label{sec.c2}
In this section, the case when
\begin{equation}
\label{sec.c2.eq1}\int_{-\infty}^\infty\mu(dx)\,x^2<\infty\,,
\end{equation}
is studied.  We also assume that
\begin{equation}
\label{sec.c2.eq2}\mu\left(\bbr\setminus\{0\}\right)>0\,,
\end{equation}
for the sake of ruling out the degenerate case when $X_t=X_0$ for all $t\in\bbr$ a.s.  An immediate consequence of \eqref{sec.c2.eq1} is that $R$ is a $C^2$ function, and hence so are the sample paths of $(X_t)$.  This in conjunction with finiteness and symmetry of $\mu$  and \eqref{sec.c2.eq2} implies that 
\[
R'(0)=0>R''(0)\,,
\]
where $R'$ and $R''$ are the first and second derivatives of $R$, respectively.  The main result of this section is the following.

\begin{theorem}\label{sec.c2.t1}
As $u\to\infty$,  conditionally on $[X_0>u]$,
\[
u\left(\tau_u^+-\tau_u^-\right)\Rightarrow2\frac{{R(0)}}{\sqrt{-R''(0)}}\sqrt{Z^2+2T^*}\,,
\]
where $Z$ and $T^*$ follow standard normal and standard exponential,  respectively, independently of each other.
\end{theorem}

\begin{proof}
All the convergences that we shall talk of in the proof, unless mentioned otherwise,  are conditionally given that $[X_0>u]$ as $u\to\infty$.  Our first claim is that
\begin{equation}
\label{sec.c2.t1.eq1} \tau_u^\pm\prob0\,,
\end{equation}
that is,  $\tau_u^+$ and $\tau_u^-$ go to zero in probability.  Proceeding towards proving this,  first note that paths of $(X_t:t\in\bbr)$ are $C^2$ because its covariance kernel $R$ is.  Fix $\vep>0$ and use Taylor's theorem to write
\begin{equation}
\label{sec.c2.t1.eq2}X_\vep=X_0+\vep X_0'+\frac12\vep^2X_\xi''\,,
\end{equation}
for some $\xi\in[0,\vep]$,  where $X'$ and $X''$ are the first and second pathwise derivatives of $X$, respectively.  Set
\begin{equation}
\label{sec.c2.t1.eq3}Y_t=X_t-\E\left(X_t|X_0\right)=X_t-\frac{R(t)}{R(0)}X_0\,, t\in\bbr\,,
\end{equation}
as a consequence of which,  \eqref{sec.c2.t1.eq2} can be rewritten as
\[
X_\vep=X_0+\vep X_0'+\frac12\vep^2\left(Y''_\xi+R(0)^{-1}R''(\xi)X_0\right)\,.
\]
Using furthermore that $R'(0)=0$,  it follows that $X_0'=Y_0'$ and hence
\[
X_\vep=X_0+\vep Y_0'+\frac12\vep^2\left(Y''_\xi+R(0)^{-1}R''(\xi)X_0\right)\,.
\]
It is known that
\begin{equation}
\label{sec.ht.t1.eq4}P\left(u(X_0-u)\in\cdot\bigr|X_0>u\right)\Rightarrow P\left(R(0)T^*\in\cdot\right)\,,u\to\infty\,,
\end{equation}
and hence
\begin{equation}
\label{sec.c2.t1.eq5}u^{-1}X_0\prob1\,.
\end{equation}
Combining this with the observation that $(Y_t:t\in\bbr)$ is independent of $X_0$, and that 
\[
\sup_{0\le t\le\vep}|Y''_t|<\infty\text{ a.s.}\,,
\]
it follows that 
\[
\lim_{u\to\infty}P\left(u^{-1}X_\vep>1+\frac12\vep^2\frac1{R(0)}\sup_{0\le t\le\vep}R''(t)\Bigr|X_0>u\right)=0\,.
\]
Since for small enough $\vep>0$ it holds that
\[
\sup_{0\le t\le\vep}R''(t)<0\,,
\]
because $R''(0)<0$ and $R$ is $C^2$,  it holds that for such $\vep$,
\[
\lim_{u\to\infty}P\left(X_\vep>u|X_0>u\right)=0\,,
\]
and hence \eqref{sec.c2.t1.eq1} holds for $\tau_u^+$.  A similar argument works for $\tau^-_u$.

Replacing $\vep$ by $\tau_u^+$ in \eqref{sec.c2.t1.eq2} yields that
\[
u=X_{\tau_u^+}=X_0+\tau_u^+X_0'+\frac12\left(\tau_u^+\right)^2X''_{\xi_u}\,,
\]
for some $\xi_u\in[0,\tau_u^+]$.
The 2 roots of the above quadratic equation in $\tau_u^+$ are 
\begin{equation}
\label{sec.c2.t1.eq6}\frac1{X''_{\xi_u}}\left(\pm \sqrt{(X_0')^2-2X''_{\xi_u}(X_0-u)}-X_0'\right)\,.
\end{equation}
Noting that
\[
X''_{\xi_u}=Y''_{\xi_u}+R(0)^{-1}R''(\xi_u)X_0\,,
\]
\eqref{sec.c2.t1.eq1} and \eqref{sec.c2.t1.eq5} imply that $\xi_u\prob0$.  Thus,
\begin{equation}
\label{sec.c2.t1.eq7}u^{-1}X''_{\xi_u}\prob\frac{R''(0)}{R(0)}<0\,,
\end{equation}
a consequence of which is that
\[
\lim_{u\to\infty}P\left(X''_{\xi_u}<0|X_0>u\right)=1\,.
\]
Since on the set $[X_0>u]\cap[X''_{\xi_u}<0]$,  it holds that
\[
\sqrt{(X_0')^2-2X''_{\xi_u}(X_0-u)}>|X_0'|\,,
\]
and $\tau_u^+$ is the positive one of the two roots in \eqref{sec.c2.t1.eq6},  it follows that
\[
\lim_{u\to\infty}P\left(\tau_u^+=-\frac1{X''_{\xi_u}}\left( \sqrt{(X_0')^2 - 2X''_{\xi_u}(X_0-u)} + X_0'\right)\Biggr|X_0>u\right)=1\,.
\]
Using \eqref{sec.ht.t1.eq4} and \eqref{sec.c2.t1.eq7},  we get
\[
X''_{\xi_u}(X_0-u)\Rightarrow R''(0)T^*\,.
\]
Combining this with the fact that $X_0'$, which equals $Y_0'$,  is independent of $X_0$, it follows that
\[
\sqrt{(X_0')^2 - 2X''_{\xi_u}(X_0-u)} + X_0' \Rightarrow \sqrt{(X_0')^2 -2 R''(0)T^*}+X_0'\,,
\]
assuming $T^*$ is independent of $(X_t:t\in\bbr)$ which is no loss of generality.  Use \eqref{sec.c2.t1.eq7} once again to get that
\[
u\tau_u^+\Rightarrow \frac{R(0)}{-R''(0)}\left(X_0'+ \sqrt{(X_0')^2 -2 R''(0)T^*} \right)\,.
\]
Recalling that for $h\neq0$,
\[
\frac{X_h-X_0}{h}\sim N\left(0,2h^{-2}(R(0)-R(h))\right)\,,
\]
and that the left hand side converges to $X_0'$ as $h\to0$,  it follows that
\[
X_0'\sim N\left(0,-R''(0)\right)\,.
\]
Thus,
\begin{align*}
u\tau_u^+ &\Rightarrow \frac{R(0)}{-R''(0)}\left(Z\sqrt{-R''(0)}+ \sqrt{-R''(0)Z^2 -2 R''(0)T^*} \right)\\
&= \frac{R(0)}{\sqrt{-R''(0)}}\left(Z+\sqrt{Z^2+2T^*}\right)\,,
\end{align*}
because $Z$ is independent of $T^*$.  The same calculations with no extra effort would show that
\[
u\tau_u^-\Rightarrow \frac{R(0)}{\sqrt{-R''(0)}}\left(Z-\sqrt{Z^2+2T^*}\right)\,,
\]
jointly with the above, from which the proof follows.
\end{proof}

\section{Spectral measure with regularly varying tails and infinite second moment}\label{sec.ht}
The assumption on the spectral measure $\mu$ in this section is that $\mu\bigl([\cdot,\infty)\bigr)$ is regularly varying with index $-\alpha$ for some $0<\alpha<2$,  that is,
\[
\mu\bigl([x,\infty)\bigr) > 0\,,\text{ for all }x>0\,,
\]
and
\begin{equation}
\label{sec.ht.eq2} \lim_{x\to\infty}\frac{\mu\bigl([cx,\infty)\bigr)}{\mu\bigl([x,\infty)\bigr)}=c^{-\alpha}\,,\text{ for all }c>0\,.
\end{equation}
This necessarily ensures that \eqref{sec.c2.eq1} fails, that is, the left hand side therein is infinite.  

Theorem 1 of \cite{pitman:1968} implies that 
\begin{equation}
\label{sec.ht.eq3} R(0) - R(t) \sim C_\alpha\mu\left(\left[t^{-1},\infty\right)\right)\,,\text{ as }t\downarrow0\,,
\end{equation}
that is,  the ratio of the quantities on the two sides of `$\sim$' goes to $1$ as $t\downarrow0$,  where 
\begin{equation}
\label{eq.defcalpha}C_\alpha= \frac{\pi}{\Gamma(\alpha)\sin(\pi\alpha/2)}\,,0<\alpha<2\,,
\end{equation}
and $\Gamma(\cdot)$ is Euler's Gamma function.  This means that \eqref{sec1.eq1} holds for any $\vep\in(0,\alpha)$.  Set
\[
h(y)=\inf\left\{x\ge0:\mu\left([x,\infty)\right)\le y\right\}\,,y>0\,,
\]
and
\[
\delta_u=\left(h(u^{-2})\right)^{-1}\,,u>0\,.
\]

For stating the main result of this section, which is Theorem \ref{sec.ht.main} below,  a few more notations are needed.  Let
\begin{equation}
\label{eq.defY}Y_t=\sqrt{2C_\alpha}B(t)+R(0)T^*-\frac1{R(0)}C_\alpha |t|^\alpha\,,t\in\bbr\,,
\end{equation}
which should not be confused with $Y_t$ defined in \eqref{sec.c2.t1.eq3},  where $T^*$ follows standard exponential,  $(B(t):t\in\bbr)$ is a fractional Brownian motion independent of $T^*$ with Hurst index $\alpha/2$, that is, it is a zero mean Gaussian process with continuous paths and covariance as follows:
\[
\Cov\left(B(s),B(t)\right)=\frac12\left(|t|^\alpha+|s|^\alpha-|t-s|^\alpha\right)\,,s,t\in\bbr\,.
\]
Set
\begin{align}
\label{eq.deftaup}\tau_*^+&=\inf\left\{t\ge0:Y_t\le0\right\}\,,\\
\label{eq.deftaum}\tau_*^-&=\sup\left\{t<0:Y_t\le0\right\}\,.
\end{align}

\begin{theorem}\label{sec.ht.main}
If $\tau_u^\pm$ are as in \eqref{sec1.eq2} and \eqref{sec1.eq3}, then 
\[
\delta_u^{-1}\left(\tau_u^+-\tau_u^-\right)\Rightarrow\tau_*^+-\tau_*^-\,,
\]
conditionally on $[X_0>u]$, as $u\to\infty$.
\end{theorem}

It follows from the definition of $h$ and \eqref{sec.ht.eq2} that 
\[
\delta_u=u^{-2/\alpha}l(u)\,,u>0\,,
\]
for some  function $l$ slowly varying at infinity.  This means that under the hypotheses of Theorem \ref{sec.ht.main},  the length of the excursion set above $u$ is roughly of order $u^{-2/\alpha}$.  This is much smaller than $u^{-1}$,  which is the order of the length of the excursion set under the assumptions of Theorem \ref{sec.c2.t1}. 

Theorem \ref{sec.ht.main} would follow almost immediately from the following result. 

\begin{theorem}\label{sec.ht.t1}
Conditional on $[X_0>u]$, as $u\to\infty$,
\[
\bigl(u\left(X_{t\delta_u}-u\right):t\in\bbr\bigr)\Rightarrow\left(Y_t:t\in\bbr\right)\,,
\]
on $C(\bbr)$ equipped with the topology of uniform convergence on compact sets.
\end{theorem}

\begin{proof}
It is known that 
\begin{equation}
\label{sec.ht.eq4}\lim_{y\downarrow0}h(y)=\infty\,,
\end{equation}
and
\begin{equation}
\label{sec.ht.eq5}\mu\left([h(y),\infty)\right)\sim y\,, y\downarrow0\,.
\end{equation}
It follows from \eqref{sec.ht.eq4} that
\[
\lim_{u\to\infty}\delta_u=0\,.
\]
Hence,  for a fixed $t\neq0$,  \eqref{sec.ht.eq3} implies that as $u\to\infty$,
\begin{align}
\nonumber \sqrt{R(0)-R(t\delta_u)} &=  \sqrt{R(0)-R(|t|\delta_u)}\\
\nonumber &\sim\left(C_\alpha\mu\left(\left[h(u^{-2})|t|^{-1},\infty\right)\right)\right)^{1/2}\\
\nonumber&\sim C_\alpha^{1/2}|t|^{\alpha/2}\sqrt{\mu\left([h(u^{-2}),\infty)\right)}\\
\label{sec.ht.t1.eq1}&\sim C_\alpha^{1/2}|t|^{\alpha/2}u^{-1}\,,
\end{align}
\eqref{sec.ht.eq5} implying the last line.  

Let
\[
Z_t=X_t-\E(X_t|X_0)\,,t\in\bbr\,.
\]
Recalling \eqref{sec.c2.t1.eq3}, it follows that for fixed $s,t\in\bbr$, 
\begin{align}\label{sec.ht.t1.eq6}
&\E\left(Z_{s\delta_u}Z_{t\delta_u}\right)\\
\nonumber&=\E\left[\left(X_{s\delta_u}-R(0)^{-1}R(s\delta_u)X_0\right)\left(X_{t\delta_u}-R(0)^{-1}R(t\delta_u)X_0\right)\right]\\
\nonumber&=R\left((s-t)\delta_u\right)-R(0)^{-1}R(s\delta_u)R(t\delta_u)\\
\nonumber&=\left[R(0)-R(s\delta_u)\right]+\left[R(0)-R(t\delta_u)\right]-\left[R(0)-R\left((s-t)\delta_u\right)\right]\\
\nonumber&\,\,\,\,\, -R(0)^{-1}\left[R(s\delta_u)-R(0)\right]\left[R(t\delta_u)-R(0)\right] \\
&\sim C_\alpha u^{-2}\left(|s|^\alpha+|t|^\alpha-|s-t|^\alpha\right)\,,\label{sec.ht.t1.eq7}
\end{align}
as $u\to\infty$ by \eqref{sec.ht.t1.eq1}.  As $(Z_t)$ is a Gaussian process, it thus follows that
\begin{equation}
\label{sec.ht.t1.eq2} \left(uZ_{t\delta_u}:t\in\bbr\right)\fdd\left(\sqrt{2C_\alpha}\,B(t):t\in\bbr\right)\,,u\to\infty\,,
\end{equation}
where `` $\fdd$ '' denotes convergence of finite-dimensional distributions.

For $t$ fixed, write
\begin{align} \nonumber
u\left(X_{t\delta_u}-u\right)&=uZ_{t\delta_u}+u\left(\frac{R(t\delta_u)}{R(0)}X_0-u\right)\\
&=uZ_{t\delta_u}+\frac{R(t\delta_u)}{R(0)}u(X_0-u)+\frac1{R(0)}u^2\left[R(t\delta_u)-R(0)\right]\,. \label{sec.ht.t1.eq3}
\end{align}
Use \eqref{sec.ht.t1.eq1} once again to note that the limit of the last term above, which is deterministic, as $u\to\infty$, is
\[
-\,\frac1{R(0)}C_\alpha|t|^\alpha\,.
\]
Recalling \eqref{sec.ht.t1.eq4} and that $(Z_t)$ is independent of $X_0$, it follows by invoking \eqref{sec.ht.t1.eq2} that conditionally on $[X_0>u]$, as $u\to\infty$,
\[
\left(u\left(X_{t\delta_u}-u\right):t\in\bbr\right)\fdd\left( \sqrt{2C_\alpha}\,B(t) + R(0)T^* -\frac1{R(0)} C_\alpha|t|^\alpha:t\in\bbr\right)\,,
\]
because  $T^*$ and $(B(t):t\in\bbr)$ are independent.  

In view of \eqref{sec.ht.t1.eq3}, \eqref{sec.ht.t1.eq4} and the independence of $X_0$ and $(Z_t)$,  for completing the proof it suffices to strengthen the convergence in \eqref{sec.ht.t1.eq2} to weak convergence on $C(\bbr)$ equipped with the topology of uniform convergence on compact sets.  As $Z_0=0$,  letting 
\[
\beta=\frac8\alpha\,,
\]
it suffices to show that given $M>0$, there exist finite positive constants $K$ and $u_0$ such that
\begin{equation}
\label{sec.ht.t1.eq5}\E\left[\left|uZ_{s\delta_u}-uZ_{t\delta_u}\right|^\beta\right]\le K(t-s)^2\,,\text{ for all }u\ge u_0\,,s,t\in[-M,M]\,.
\end{equation}
Proceeding as in \eqref{sec.ht.t1.eq6}-\eqref{sec.ht.t1.eq7},  it can be shown that for all $s,t,u$,
\begin{align}\nonumber
\Var\left(Z_{s\delta_u}-Z_{t\delta_u}\right)
&=2\left[R(0)-R((s-t)\delta_u)\right]-\frac1{R(0)}\left[R(s\delta_u)-R(t\delta_u)\right]^2\\
&\le2\left[R(0)-R((s-t)\delta_u)\right]\,.\label{sec.ht.t1.eq8}
\end{align}
Define a function $U:(0,\infty)\to(0,\infty)$ by
\[
U(x)=R(0)-R(1/x)\,,
\]
which is regularly varying with index $-\alpha$ by \eqref{sec.ht.eq2} and \eqref{sec.ht.eq3}.  Proposition 2.6 on pg 32 of \cite{resnick:2007} implies that there exists $v_0>0$ such that
\[
\sup_{v\ge v_0,x\ge1/2M}\frac{U(vx)}{U(v)}\le2x^{-\alpha/2}\,.
\]
If $u_0$ is such that 
\[
\delta_u\le\frac1{v_0}\,,\text{ for all }u\ge u_0\,,
\]
then a restatement of the above is
\[
R(0)-R\left((t-s)\delta_u\right)\le2(t-s)^{\alpha/2}\left[R(0)-R(\delta_u)\right]\,,u\ge u_0,-M\le s<t\le M\,,
\]
which is obtained by putting $v=1/\delta_u$ and $x=1/(t-s)$.

Using \eqref{sec.ht.t1.eq1}, a finite constant $K'$ can be obtained such that
\[
R(0)-R(\delta_u)\le K'u^{-2}\,,u\ge u_0\,.
\]
Thus,
\[
R(0)-R\left((t-s)\delta_u\right)\le2K'u^{-2}(t-s)^{\alpha/2}\,,u\ge u_0,-M\le s<t\le M\,.
\]
The fact that $Z_{s\delta_u}-Z_{t\delta_u}$ follows normal with mean zero and variance upper bounded by \eqref{sec.ht.t1.eq8} implies the existence of a constant $K''$ such that for $-M\le s<t\le M$ and $u\ge u_0$,
\begin{align*}
\E\left[\left|Z_{s\delta_u}-Z_{t\delta_u}\right|^\beta\right] & \le K''\left[R(0)-R((t-s)\delta_u)\right]^{\beta/2}\\
&\le K''(2K')^{\beta/2}u^{-\beta}(t-s)^{\alpha\beta/4}\\
&=Ku^{-\beta}(t-s)^2\,,
\end{align*}
where $K=K''(2K')^{\beta/2}$.  As this is equivalent to \eqref{sec.ht.t1.eq5}, the proof follows.
\end{proof}

\begin{proof}[Proof of Theorem \ref{sec.ht.main}]
Note that the function $\tau^+:C(\bbr)\to[0,\infty]$ defined by
\[
\tau^+(f)=\inf\{t\ge0:f(t)\le0\}\,,
\]
is a.s.\ finite and continuous at a sample path of $(Y_t:t\in\bbr)$,  where $C(\bbr)$ is endowed with the topology induced by uniform convergence on compact sets, and so is $\tau^-:C(\bbr)\to[0,\infty]$ defined by
\[
\tau^-(f)=\sup\{t\le0:f(t)\le0\}\,.
\]
In view of this, the proof follows from Theorem \ref{sec.ht.t1} with the help of the continuous mapping theorem.
\end{proof}

\section{Example}\label{sec.ex} In this section,  the results of Sections \ref{sec.c2} and \ref{sec.ht} are applied to an example.  Consider a zero mean stationary continuous Gaussian process $(X_t:t\in\bbr)$ with 
\[
\Cov(X_s,X_t)=e^{-|s-t|^\alpha}\,,s,t\in\bbr\,,
\]
for some $0<\alpha\le2$.   Such a process exists for all $\alpha\in(0,2]$ because 
\[
R(t)=e^{-|t|^\alpha}\,, t\in\bbr\,,
\]
is a valid characteristic function, namely of the symmetric $\alpha$-stable distribution.  It is easy to see that when $\alpha=2$,  the assumptions of Theorem \ref{sec.c2.t1} are satisfied.  A consequence of that is as $u\to\infty$, conditionally on $[X_0>u]$,
\begin{equation}
\label{sec.ex.eq1}u\left(\tau^+-\tau^-\right)\Rightarrow\sqrt{2(Z^2+2T^*)}\,,\text{ if }\alpha=2\,,
\end{equation}
where $Z$ and $T^*$ are as in the statement of Theorem \ref{sec.c2.t1}.

It is known that \eqref{sec.ht.eq2} holds when $0<\alpha<2$; see for example \cite{samorodnitsky:taqqu:1994}.  In fact, if $\mu$ is the measure satisfying \eqref{eq.defR},  then it holds that
\[
\mu\left([x,\infty)\right)\sim C_\alpha^{-1}x^{-\alpha}\,,x\to\infty\,,0<\alpha<2\,,
\]
where $C_\alpha$ is as in \eqref{eq.defcalpha}.  A comparison with \eqref{sec.ht.eq5} yields that
\[
h(y)\sim\left(C_\alpha y\right)^{-1/\alpha}\,,y\downarrow0\,,
\]
and hence
\[
\delta_u\sim C_\alpha^{1/\alpha}u^{-2/\alpha}\,,u\to\infty\,.
\]
Thus, the claim of Theorem \ref{sec.ht.main} boils down in this case to
\begin{equation}
\label{sec.ex.eq3}    u^{2/\alpha}\left(\tau^+-\tau^-\right) \Rightarrow C_\alpha^{1/\alpha}\left(\tau^+_*-\tau^-_*\right)\,,0<\alpha<2\,,
\end{equation}
conditionally on $[X_0>u]$ as $u\to\infty$,  where $\tau_*^\pm$ are the hitting times of $(Y_t)$ as in \eqref{eq.deftaup} and \eqref{eq.deftaum}, and the latter is as defined in \eqref{eq.defY}.  Since $(B(t):t\in\bbr)$ therein is a fractional Brownian motion with Hurst index $\alpha/2$, it holds that
\[
\left(C_\alpha^{1/2}B(t):t\in\bbr\right)\eid\left(B\left(C_\alpha^{1/\alpha}t\right):t\in\bbr\right)\,.
\]
Therefore,
\begin{align*}
\tau_*^+&\eid\inf\left\{t\ge0:\sqrt2B\left(C_\alpha^{1/\alpha}t\right)+T^*-\left(C_\alpha^{1/\alpha}t\right)^\alpha\le0\right\}\\
&=C_\alpha^{-1/\alpha}\inf\left\{s\ge0:\sqrt2B(s)+T^*-s^\alpha\le0\right\}\,,
\end{align*}
where $T^*$, as before, follows standard exponential independently of $(B(t))$, recalling that $R(0)=1$ in this example. A similar calculation holds for $\tau_*^-$ jointly with the above. Plugging these in \eqref{sec.ex.eq3} yields that for $0<\alpha<2$, conditionally on $[X_0>u]$ as $u\to\infty$,
\begin{equation}
\label{sec.ex.eq2} u^{2/\alpha}\left(\tau^+-\tau^-\right)\Rightarrow\tilde\tau^+-\tilde\tau^-\,,
\end{equation}
where 
\begin{align*}
\tilde\tau^+&=\inf\left\{t\ge0:\tilde Y_t\le0\right\}\,,\\
\tilde\tau^-&=\sup\left\{t<0:\tilde Y_t\le0\right\}\,,
\end{align*}
and
\begin{equation}\label{sec.ex.eq4}
\tilde Y_t=\sqrt2B(t)+T^*-|t|^\alpha\,,t\in\bbr\,.
\end{equation}

Combining \eqref{sec.ex.eq2} with \eqref{sec.ex.eq1} shows that the length of the excursion set above $u$ is of the order $u^{-2/\alpha}$ for all $\alpha\in(0,2]$.  That is,  the order of the length of the excursion set increases with increase in $\alpha$, which is not surprising because the paths of $(X_t)$ become smoother as $\alpha$ increases.  In fact, for $\alpha=2$, the process has analytic paths.

It is worth noting that when $\alpha=1$, $(X_t)$ is an Ornstein-Uhlenbeck process, and $(B(t))$ in \eqref{sec.ex.eq4} is a two-sided standard Brownian motion with $B(0)=0$. In this case, the independence of $(B(s):s\ge0)$ and $(B(s):s\le0)$ is heuristically consistent with the fact that $(X_t)$ is a Markov process.

\section*{Acknowledgment} The authors are grateful to Gennady Samorodnitsky for helpful discussions. Sukrit Chakraborty's research was supported by the NBHM postdoctoral fellowship.

\bibliographystyle{apalike}   
%\bibliography{bibfile}   

\end{document}